\documentclass[a4paper,reqno,11pt]{amsart}
\usepackage{amsmath,amssymb}
\usepackage{amsthm}
\usepackage{amscd}
\usepackage{amssymb}
\usepackage{amsfonts}
\usepackage{latexsym}
\usepackage[mathscr]{eucal}

\newtheorem{thm}{Theorem}[section]
\newtheorem{prop}[thm]{Proposition}
\newtheorem{lemma}[thm]{Lemma}
\newtheorem{cor}[thm]{Corollary}
\newtheorem{defn}[thm]{Definition}
\newtheorem{exam}[thm]{Example}

\newtheorem{rem}[thm]{Remark}

\begin{document}


\begin{center}
\LARGE \textbf{Perturbations of intermediate C$^*$-subalgebras for simple C$^*$-algebras}
\end{center}

\

\begin{center}
\large SHOJI INO AND YASUO WATATANI
\end{center}

\

\noindent \textsc{Abstract.}
We study uniform perturbations of intermediate C$^*$-subalgebras 
of inclusions of simple C$^*$-algebras. 
If a unital simple C$^*$-algebra has a simple 
C$^*$-subalgebra of finite index, then sufficiently close simple 
intermediate C$^*$-subalgebras are unitarily equivalent. These 
C$^*$-subalgebras need not to be nuclear.  
The unitary can be chosen in the relative commutant algebra. 
An imediate corollary is the following: 
If the relative commutant is trivial, then 
the set of intermediate C$^*$-subagebras is a finite set. 

\section{Introduction}

 The study of the uniform perturbation theory of operator algebras 
was started with \cite{Kadison} by Kadison and Kastler in 1972.  
They introduced a metirc on  the set of C$^*$-subalgebras of 
a fixed C$^*$-algebra 
by the Hausdorff distance between the unit balls. 
A conjugacy by a unitary near to the identity gives close C$^*$-subalgebras. 
They conjectured that suitably close C$^*$-algebras 
must be unitarily equivalent. Although Choi and Christensen gave counterexamples 
to the conjecture in \cite{Choi}, 
the conjecture has been verified in various situations. 
The problem was solved positively when one algebra is separable and AF 
in Chiristensen \cite{Chris3} and separable C$^*$-algebras of 
continuous trace by Phillips-Raeburn \cite{Phillips}. Khoshkam 
\cite{Khoshkam} showed that sufficiently close nuclear C$^*$-algebras 
have isomorphic K-groups. 
Recently Christensen, Sinclair, Smith, White and Winter 
have solved it completely  when one algebra is separable and nuclear 
in \cite{Chris4}.

In this paper 
we study uniform perturbations of intermediate C$^*$-subalgebras of 
inclusions of simple C$^*$-algebras. 
If a unital simple C$^*$-algebra has a simple 
C$^*$-subalgebra of finite index, then sufficiently close simple
intermediate C$^*$-subalgebras are unitarily equivalent.
Thanks to Izumi's result in \cite{Izumi}, we do not need to 
assume the existence of the conditional expectations 
onto intermediate C$^*$-subalgebras apriori. 
Furthermore, if the relative commutant is trivial, then 
the set of intermediate C$^*$-subagebras is a finite set.

In the case of subfactor theory of  Jones \cite{Jones}, 
the lattices of intermediate subfactors with  their finiteness 
were studied  in   Popa \cite{Po}, 
Watatani \cite{Watatani2}, Teruya-Watatani \cite{TW}, 
Longo \cite{Longo}, Khoshkam-Mashhood \cite{Kh-Ma}, Grossman-Jones \cite{G-J}, 
Grossman-Izumi \cite{G-I} and  Xu \cite{Xu} for example. In these study, 
the $\| \cdot\|_2$-perterbation technique of 
von Neumann algebras developed by 
Christensen \cite{Chris2} are essentially used.

We prepare to consider {\it uniform} perturbations of 
intermediate C$^*$-subalgebras 
of inclusions of not necessarily simple C$^*$-algebras. 
But in general, we need to assume the existance of 
the conditional expectations 
onto intermediate C$^*$-subalgebras to get a positive solution 
for the moment.

\section{Metric and index for C$^*$-subalgebras}

Let $C$ be a C$^*$-algebra. We denote by $C_1$ the unit ball of $C$. 
We recall a metric on the set of all subalgebras of a C$^*$-algebra $C$  
after Kadison and Kastler \cite{Kadison}. 
Let $A$ and $B$ be C$^*$-subalgebras of $C$. Then 
the distance $d(A,B)$ between $A$ and $B$ is defined by 
the Hausdorff distance between the unit balls of $A$ and $B$, that is, 
$$
d(A,B) = \max\left\{\sup_{a \in A_1} \inf_{b \in B_1} \|a-b\| \ , 
 \               \sup_{b \in B_1} \inf_{a \in A_1} \|b-a\| \right\}.
$$
Therefore if $d(A,B)<\gamma$, then for each $x$ in the unit ball of either $A$ or $B$, there exists $y$ in the unit ball of the other algebra with
$\|x-y\|<\gamma$.

We need the following known fact to show that a desired inclusion is onto. 

\begin{lemma}\label{lem:<1}
Let $A$ and $B$ be $\mathrm{C}^*$-subalgebras of a $\mathrm{C}^*$-algebra. 
If $A\subset B$ and $d(A,B)<1$, then $A=B$. 
\end{lemma}

The next proposition records some standard estimates.

\begin{lemma}\label{lem:estimates}
Let $A$ be a unital $\mathrm{C}^*$-algebra.
\begin{enumerate}
\item 
Let $x \in A$ satisfy that $\|x-1\|<1$ and 
$u \in A$ be the unitary in the polar decomposition $x = u|x|$. 
Then
\[ 
\|u-1\|\le\sqrt{2}\|x-1\|. 
\]
\item
Let $p \in A$ be a projection and $a \in A$ be self-adjoint. 
Assme that $\delta := \|a- p\| < 1/2$. 
Let $q = \chi_{[1-\delta,1+\delta]}(a)$. Then $q$ is a projection 
in $C^*(a,I)$  such that $\|q-p\| \leq 2\|a- p\|< 1$.
\item
Let $p$ and $q$ be projections in $A$ with $\|p-q\|<1$. 
Then there exists a unitary $w\in A$ such that 
\[
wpw^*=q,  \ \ \text{ and } \ \  
\|w-1\|\le \sqrt{2}\|p-q\|.
\]
\end{enumerate}
\end{lemma}

We recall some basic facts on index for C$^*$-subalgebas in \cite{Watatani1}.

\begin{defn}\upshape
Let $B$ be a C$^*$-algebra and $A$ a C$^*$-subalgebra of 
$B$ with a common unit. Let $E$ be a conditional expectation of $B$ onto $A$. 
We say that $E$ is {\it of finite index} 
if there exists a finite set $\{u_1,\dots,u_N\}\subset B$, 
called a \textit{(quasi-)basis} for $E$, such that
\[ b=\sum_{i=1}^N u_iE(u_i^*b) \ , \ \ \text{ \ for any \  }  b\in B. \] 
When $E$ is of finite index, then  the index of $E$ is defined by
\[ \mathrm{Index}\,E=\sum_{i=1}^N u_iu_i^* . \] 
\end{defn}

The value $\mathrm{Index}\,E$ is in the center of $B$ and 
does not depend on the choice of a basis for $E$. Moreover 
$\mathrm{Index}\,E$ is positive invertible operator in $B$. 
In fact $\mathrm{Index}\,E \geq I$.

We can choose a basis for $E$  in the unit ball of $B$ if it is necessary. 
In fact, choose a positive integer $K$ such that 
 $K \geq \max \{\|u_1\|, \dots, \|u_N\| \}$.  Define 
\[
v_{(i-1)K+j} = \frac{1}{\sqrt{K}}u_i, \ \ \text{ \ for \ }
i = 1,2,\dots, N, \ j= 1,2, \dots, K. 
\]
Then $\{v_1, \dots, v_{KN}\}$ is a desired basis. 

Next, we recall the C$^*$-basic construction. 
Let $E:B\to A$ be a faithful conditional expectation. 
Define a $A$-valued inner product on $B$ by
\[ \langle x, y\rangle_A=E(x^*y) \ , \ \ x,y\in B .\]
We denote by  $\mathcal{E}$  the completion of $B$. Then 
$\mathcal{E}$ becomes a 
Hilbert $A$-module. Let $\eta: B \rightarrow \mathcal{E}$ be 
the natural inclusion map. Thus $\|\eta(x)\| = \|E(x^*x)\|^{1/2}$. 
Let $\mathcal{L}_A(\mathcal{E})$ 
be the set of bounded $A$-module maps on $\mathcal{E}$ with adjoints. 
Let $\mathcal{K}_A(\mathcal{E})$ be the set of ``compact''
operators on $\mathcal{E}$.

For $b\in B$, define $\lambda(b)\in \mathcal{L}_A(\mathcal{E})$ by
\[ \lambda(b)\eta(x)=\eta(bx) \ , \ \ x\in B. \]
Then $\lambda:B\to \mathcal{L}_A(\mathcal{E})$ turns out to be an injective $*$-homomorphism. Define the Jones projection 
$e_A\in \mathcal{L}_A(\mathcal{E})$ by
\[
 e_A\eta(x)=\eta(E(x)) \ , \ \ x\in B .
\]
The C$^*$-basic construction $C^*\langle B,e_A\rangle$ is defined by 
the closure of the linear span of $\{\lambda(x)e_A\lambda(y) | x,y\in B\}$.

Moreover $\lambda$ and $e_A$ satisfy the following: 
\begin{enumerate}
\item (covariant relation)  $e_A\lambda(b)e_A=\lambda(E(b))e_A$ \ for $b\in B$. \item  Let $b\in B$. Then $b$ is in $A$ if and only 
if $e_A\lambda(b)=\lambda(b)e_A$. 
\end{enumerate}

\begin{defn} \upshape
Let $D$ be a $\mathrm{C}^*$-algebra and $C$ a $\mathrm{C}^*$-subalgebra of $D$ with a common unit.  Assume that the inclusion $C \subset D$ has a conditional expectation 
$E_C^D:D\to C$ of finte index. 
We denote by  $\mathrm{IMS}(C,D,E_C^D)$ the set of all intermediate 
$\mathrm{C}^*$-subalgebra $A$ between $C$ and $D$  with  
a conditional expectation $E_A^D:D\to A$ 
satisfying the compatibility condition $E_C^A\circ E_A^D=E_C^D$, 
where $E_C^A:= E_C^D|_A:A\to C $ is the coditional expectation defined by 
the restriction of $E_C^D$ to $A$. 
\end{defn}

Let $A$ be in $\mathrm{IMS}(C,D,E_C^D)$. If 
there exists another conditional expectation 
$F_A^D:D\to A$ 
satisfying the compatibility condition $E_C^A\circ F_A^D=E_C^D$, 
then $F_A^D = E_A^D$. In fact for any $x \in D$ and $a \in A$, 
we have that $E_C^D(F_A^D(ax))=E_C^D(ax)$. Hence 
\[
E_C^D(aF_A^D(x))=E_C^D(ax)= E_C^D(aE_A^D(x)). 
\]
Then $E_C^D(a(F_A^D(x)-E_A^D(x)))=0.$  
Put $a = (F_A^D(x)-E_A^D(x))^*$. Since $E_C^D$ is 
of finite index, $E_C^D$ is faithful. 
This implies that $F_A^D = E_A^D$.

For any $A$ in $\mathrm{IMS}(C,D,E_C^D)$, the conditional expectation 
$E_C^A=E_C^D|_A:A\to C $ is of finite index.  In fact, 
let $\{u_1,\dots,u_N\}\subset D$ be a basis of $E_C^D$. 
Put $v_i = E_A^D(u_i)$. Then 
$\{v_1,\dots,v_N\}\subset D$ is a basis of $E_C^A$. 
If $u_i$ is in  $D_1$, then $v_i$ is in $A_1$.

But we should be careful that an intermediate 
C$^*$-subalgebra $A$ between $C$ and $D$  may not have 
a conditional expectation $E_A^D:D\to A$ in general.

\begin{exam} \upshape
Let $D = C([0,1], M_2({\mathbb C}))$ be the 
algebra of $2 \times 2$ matrix valued continuous functions 
on the unit interval $[0,1]$. Consider a C$^*$-subalgebra 
$C = C([0,1],{\mathbb C}I) \subset D$. Then there exist a 
conditional expectation $E_C^D:D\to C$ of finte index. In fact, 
define  $(E_C^D(f))(x) = tr(f(x))I$ for $f\in D$ and $x \in [0,1]$, where 
$tr$ is the normalized trace on $M_2({\mathbb C})$. 
Define an intermediate C$^*$-subalgebra $A$ between $C$ and $D$ 
by 
\[
A := \left\{ f \in C([0,1], M_2({\mathbb C})) \ \Big| \ 
      f\left(\frac{1}{2}\right) \in {\mathbb C}I \right\}
\]
Then there exist no conditional expectation of $D$ onto $A$. 
In fact, suppose that there were a conditional expectation 
$E_A^D:D\to A$. 
We can choose $a \in C([0,1], {\mathbb C})$ 
such that $a(1/2)=0$ and $a(x)\neq 0$ for any $x\neq 1/2$. 
Since $\begin{pmatrix} a & 0 \\ 0 & 0 \end{pmatrix} , \begin{pmatrix} a & 0 \\ 0 & a \end{pmatrix} \in A$, we have 
\[
\begin{pmatrix} a & 0 \\ 0 & 0 \end{pmatrix}=
E_A^D\left(\begin{pmatrix} a & 0 \\ 0 & 0 \end{pmatrix}\right)=
E_A^D\left(\begin{pmatrix} {\bf 1} & 0 \\ 0 & 0 \end{pmatrix}\begin{pmatrix} a & 0 \\ 0 & a \end{pmatrix}\right)=
E_A^D\left(\begin{pmatrix} {\bf 1} & 0 \\ 0 & 0 \end{pmatrix} \right) \begin{pmatrix} a & 0 \\ 0 & a \end{pmatrix},
\]
where ${\bf 1}$ is an identity of $C([0,1], \mathbb{C})$. This shows that $E_A^D\left( \begin{pmatrix} {\bf 1} & 0 \\ 0 & 0 \end{pmatrix}\right)(x)=\begin{pmatrix} 1 & 0 \\ 0 & 0 \end{pmatrix}$ for any $x\neq 1/2$. 
By the continuity, $E_A^D\left( \begin{pmatrix} {\bf 1} & 0 \\ 0 & 0 \end{pmatrix}\right)=\begin{pmatrix} {\bf 1} & 0 \\ 0 & 0 \end{pmatrix}$.
But $\begin{pmatrix} {\bf 1} & 0 \\ 0 & 0 \end{pmatrix}\notin A$.
This is a contradiction.  
\end{exam}

There is a relation between $d(A,B)$ and  the norm 
estimate $\|e_A-e_B\|$ of Jones projections $e_A$ and $e_B$ 
for intermediate C$^*$-subalgebras $A,B\in\mathrm{IMS}(C,D,E_C^D)$. 
More precisely, let $\mathcal{E}$  be the completion of $D$ by 
the $C$-valued inner procuct $E_C^D(x^*y)$. Then  
$\mathcal{E}$ becomes a 
Hilbert $C$-module. Let $\eta: D \rightarrow \mathcal{E}$ be 
the natural inclusion map. We can define Jones projections 
for intermediate $C^*$-subalgebras by 
$e_A\eta(x) = \eta(E_A^D(x))$ and 
$e_B\eta(x) = \eta(E_B^D(x))$ for $x \in D$. These projections also 
enjoy  similar properties with a usual Jones projection $e_C$. 
For example, $e_A$ commute with the left multiplication 
operator $\lambda(a)$ for $a \in A$. 

\begin{lemma}\label{lem:Jones projection metric}
Let $D$ be a $\mathrm{C}^*$-algebra and $C$ a $\mathrm{C}^*$-subalgebra of $D$ with a common unit and $E_C^D:D\to C$ a conditional expectation of finite index. 
Then for $A,B\in\mathrm{IMS}(C,D,E_C^D)$ we have that 
\[ d(A,B)\le \|\mathrm{Index}\, E_C^D\| \|e_A-e_B\|. \]
\end{lemma}
\begin{proof}
Put $c=\|\mathrm{Index}\, E_C^D\|$. Then for any $a\in A_1$, 
we have $\|\eta(a)\| \leq \|a\| \leq 1$. By the 
Pimsner-Popa inequality, $E(x^*x) \geq c^{-1}x^*x$ of Proposition 2.6.2 
in  \cite{Watatani1}, 
\begin{align*}
\|e_A-e_B\| &\ge \frac{\|\eta(E_A^D(a)-E_B^D(a))\|}{\|\eta(a)\|} 
 \ge \|\eta(a-E_B^D(a))\| \\
 &=\|E_C^D((a-E_B^D(a))^*(a-E_B^D(a)))\|^{1/2} \\
 &\ge \frac{1}{c}\|(a-E_B^D(a))^*(a-E_B^D(a))\|^{1/2} 
 =\frac{1}{c}\|a-E_B^D(a)\|.
\end{align*}
Therefore for any $a\in A_1$, we can find $b := E_B^D(a) \in B_1$ 
such that $\|a-b\| \leq c\|e_A-e_B\|$. By a symmmetric argument, we 
have that $d(A,B)\le c \|e_A-e_B\|$. 
\end{proof}

\section{Perturbations}

We begin with  some elementary estimations. 

\begin{lemma}\label{lem:multiplicativity}
Let $A$ and $D$ be $\mathrm{C}^*$-algebras. Let $\varphi: A \rightarrow D$ be 
a contractive positive map and $\psi: A \rightarrow D$ be 
a $*$-homomorphism. Then for any $x,y \in A$ 
\[
\|\varphi(xy) - \varphi(x)\varphi(y)\| \leq 3\|\varphi- \psi\|\|x\|\|y\|.
\]
\end{lemma}
\begin{proof}
Approximate $\varphi$ by $\psi$. 
\end{proof}

\begin{lemma}\label{lem:E and i}
Let $D$ be a  $\mathrm{C}^*$-algebra and $A,B$ be  $\mathrm{C}^*$-subalgebras of $D$. 
Let $E_B: D \rightarrow B$ be a conditional expectation. 
Consider the restriction map 
$E_B|_A: A  \rightarrow D$ and an inclusion map 
$\iota_A: A \rightarrow D$. Then we have
\[
\| E_B|_A - \iota_A\| \leq 2d(A,B), 
\]
and for any $x,y \in A$,
\[
\| E_B(xy) - E_B(x)E_B(y)\| \leq 6d(A,B)\|x\|\|y\|.
\]
\end{lemma}
\begin{proof}
For any $\epsilon > 0$ and $a \in A_1$, there exists $b \in B_1$ such 
that $\|a-b\| \leq d(A,B) + \varepsilon/2$. Then
\[
\|E_B(a) -\iota_A(a) \| \leq \|E_B(a-b)\| + \|a-b\| \leq 2d(A,B) + \varepsilon.
\]
Hence $\| E_B|_A - \iota_A\| \leq 2d(A,B)$. The rest is clear. 
\end{proof}

We shall show  that two close intermediate C$^*$-subalgebras 
$A,B\in \mathrm{IMS}(C,D,E_C^D)$ are unitarily equivalent. 
We need the following two key lemmas:  

\begin{lemma}
Let $D$ be a $\mathrm{C}^*$-algebra and $C$ a $\mathrm{C}^*$-subalgebra of $D$ with a common unit. Let $E_C^D:D\to C$ be a conditional expectation of finite index
 with a basis $\{u_1,\dots,u_N\}$ in $ D_1$. 
For any $A,B\in\mathrm{IMS}(C,D,E_C^D)$, if $d(A,B)<(24N^2)^{-1}$, 
then there exists a unital $*$-homomorphism $\psi:A\to B$ 
such that $\psi|_C=id_C$ and
\[ \|E_B^D|_A-\psi\|\le 8\sqrt{3}N\sqrt{d(A,B)}. \] 
\end{lemma}

\begin{proof}
Let $\mathcal{E}$ be the Hilbert $B$-module completion of $D$ 
using $E_B^D$ 
and $\eta$ the natural inclusion map from $D$ to $\mathcal{E}$.
 Define an injective $*$-homomorphism $\lambda:D\to \mathcal{L}_C(\mathcal{E})$ by 
\[ 
\lambda(d)\eta(x)=\eta(dx) \ , \ \ d,x\in D. 
\]
Then for any $b \in B$, we have  $\lambda(b)e_B = e_B\lambda(b)$.  
The map $B \ni b \mapsto \lambda(b)e_B\in \mathcal{L}_C(\mathcal{E})$ is 
a injective $*$-homomorphism and 
$\|\lambda(b)e_B\| = \|\lambda(b)\|$. 
For any $z \in D$, we have $e_B\lambda(z)e_B = E_B^D(z)e_B$. 
Hence 
for any $x,y \in D$, 
\begin{align*}
& \|E_B^D(xy) -E_B^D(x)E_B^D(y)\| 
 = \|\lambda(E_B^D(xy) -E_B^D(x)E_B^D(y))e_B\|  \\
& =\|\lambda(E_B^D(xy))e_B - \lambda(E_B^D(x))e_B\lambda(E_B^D(y))e_B\| \\
& =\| e_B\lambda(xy)e_B - e_B\lambda(x)e_B\lambda(y)e_B\|
= \| e_B\lambda(x)(I - e_B)\lambda(y)e_B \|. 
\end{align*}
Let $E_C^A$ be a restriction of $E_C^D$ to $A$  and put $v_i=E_A^D(u_i)$.
Then $\{v_1,\dots,v_N\}$ is a basis for $E_C^A$ in $A_1$. Define an operator 
\[ 
t=\sum_{i=1}^N \lambda((\mathrm{Index}\,E_C^A)^{-1}) \lambda(v_i)e_B\lambda(v_i^*) \in \mathcal{L}_C(\mathcal{E}). 
\]
Recall that $\mathrm{Index}\,E_C^A$ is in the center of $A$. 
For any $a\in A$ 

\begin{align*}
\lambda(a)t &=\sum_{i=1}^N \lambda(( \mathrm{Index}\, E_C^A)^{-1})\lambda(av_i)e_B\lambda(v_i^*) \\
 &=\sum_{i=1}^N\sum_{j=1}^N \lambda(( \mathrm{Index}\,E_C^A)^{-1}) \lambda(v_j E_C^A(v_j^*av_i))e_B\lambda(v_i^*) \\
 &=\sum_{j=1}^N \lambda(( \mathrm{Index}\, E_C^A)^{-1}) \lambda(v_j)e_B \sum_{i=1}^N \lambda(E_C^A(v_j^*av_i)v_i^*) \\
 &=\sum_{j=1}^N \lambda(( \mathrm{Index}\, E_C^A)^{-1})\lambda(v_j)e_B\lambda(v_j^*)\lambda(a)=t\lambda(a).   
\end{align*}
Since $\mathrm{Index}\, E_C^A = \sum_iv_iv_i^*$ and Lemma \ref{lem:E and i}, 
\begin{align*}
\|t-e_B\| &=\left\|t-\sum_{i=1}^N \lambda((\mathrm{Index}\, E_C^A)^{-1})\lambda(v_i)\lambda(v_i^*)e_B\right\| \\
 &\le\sum_{i=1}^N \|\lambda(( \mathrm{Index}\, E_C^A)^{-1})\| \|\lambda(v_i)\| \|e_B\lambda(v_i^*)-\lambda(v_i^*)e_B\| \\
 &\le N\|e_B\lambda(v_i^*)-\lambda(v_i^*)e_B\| \\
 &= N\|(I-e_B)\lambda(v_i^*)e_B-e_B\lambda(v_i^*)(I-e_B)\| \\
 &=N\cdot\max\left\{ \|(I-e_B)\lambda(v_i^*)e_B\| \ , \ \|e_B\lambda(v_i^*)(I-e_B)\| \right\} \\
 &=N\cdot \max
\left\{ \|e_B\lambda(v_i)(I-e_B)\lambda(v_i^*)e_B\|^{1/2}  \ , \ \|e_B\lambda(v_i^*)(I-e_B)\lambda(v_i)e_B\|^{1/2} \right\} \\
 &=N\cdot \max \left\{\|E_B^D(v_iv_i^*)-E_B^D(v_i)E_B^D(v_i^*)\|^{1/2},  \ 
\|E_B^D(v_i^*v_i)-E_B^D(v_i^*)E_B^D(v_i)\|^{1/2} \right\}   \\ 
 &\le N\sqrt{6d(A,B)}<\frac{1}{2}. 
\end{align*}

Put  $\delta = \|t-e_B\| = \|(t + t^*)/2 -e_B\|$. 
Let $q=\chi_{[1-\delta, 1 + \delta]}((t+t^*)/2)$.  
By Lemma \ref{lem:estimates}, 
$q$ is a projection in $C^*(t,I_{\mathcal{E}})$ and  
commutes with $\lambda (a)$ for any $a \in A$,  and  
$\|q-e_B\|\le2\|t-e_B\|<1$. Therefore  there exists a unitary
 $w\in C^*(t, e_B, I_{\mathcal{E}})\subset \mathcal{L}_B(\mathcal{E})$ such that $we_Bw^*=q$ and $\|w-1_{\mathcal{E}}\|\le \sqrt{2}\|q-e_B\|$.  
Define $\psi^{\prime}:A\to \mathcal{L}_B(\mathcal{E})$ by 
\[ 
\psi^{\prime}(a)= w^*q\lambda(a)qw = e_Bw^*\lambda(a)we_B \ , \ \ a\in A .
\] 
Since the projection $q$ commutes with $\lambda (a)$, it is clear that 
$\psi^{\prime}$ is a $*$-homomorphism. 

The unitary  $w$ is  in $C^*(t, e_B, I_{\mathcal{E}})\subset C^*(\lambda(A), e_B, I_\mathcal{E})$ , and $e_Bw^*\lambda(a)we_B\in \lambda(B)e_B$ for $a\in A$. Therefore, $\psi^{\prime}(A)\subset \lambda(B)e_B\subset \mathcal{L}_B(\mathcal{E})$. Let $\mathcal{E}^{\prime}$ be a closure of $\eta(B)$ in $\mathcal{E}$. Define an injective $*$-homomorphism $\lambda^{\prime}:B\to \mathcal{L}_B(\mathcal{E}^{\prime})$ by
\[  \lambda^{\prime}(b)\eta(x)=\eta(bx) \ , \ \ b,x\in B, \]
and a surjective $*$-isomorphism $\iota:\lambda(B)e_B\to \lambda^{\prime}(B)$ by
\[ \iota(\lambda(b)e_B)=\lambda^{\prime}(b) \ , \ \ b\in B. \]
Then $\lambda(B)e_B$ is isomorphic to $\lambda^{\prime}(B)$. Thus, 
we can define a $*$-homomorphism $\psi=(\lambda^{\prime})^{-1}\circ \iota \circ\psi^{\prime}:A\to B$, that is,  
$\lambda(\psi(a))e_B = \psi^{\prime}(a)$.   Then for any contraction $a \in A$, 
\begin{align*}
\|E_B^D|_A(a)- \psi(a)\|
 & = \|e_B(\lambda(E_B^D|_A(a)- \psi(a)))e_B\| 
= \|e_B\lambda(E_B^D(a))e_B- \psi^{\prime}(a)\| \\
& = \| e_B\lambda(a)e_B  -  e_Bw^*\lambda(a)we_B \| \\
& = \| e_B\lambda(a)(I - w^*)e_B  + e_B(I - w^* )\lambda(a)we_B \|
\leq 2\|w-I\|. 
\end{align*}
Therefore 
\[ 
\|E_B^D|_A- \psi\|\le 2\|w-I\| \le 8\sqrt{3}N\sqrt{d(A,B)}. 
\]

Since $w$ is contained in $C^*(t, e_B ,I_{\mathcal{E}})$, $C \subset A$ 
and $C \subset B$, we have that  
$\lambda(c)w=w\lambda(c)$ for $c \in C$. Hence 
\[
\psi^{\prime}(c)= e_Bw^*\lambda(c)we_B = e_B\lambda(c)e_B = \lambda(c)e_B
\]
Therefore  $\psi(c) = c = id_C(c)$ for any $c \in C$. 
\end{proof}

Next Lemma shows how to find an intertwiner for close $*$-homomorphisms. 

\begin{lemma}
Let $D$ be a $\mathrm{C}^*$-algebra and $C$ a $\mathrm{C}^*$-subalgebra of $D$ with a common unit and $E_C^D:D\to C$ be a conditional expectation of finite index with a basis $\{u_1,\dots,u_N\}$ in $D_1$.  For any $A\in\mathrm{IMS}(C,D,E_C^D)$, If $\phi_1,\phi_2:A\to D$ are unital $*$-homomorphisms such that
 $\phi_1|_C=id_C=\phi_2|_C$ and $\|\phi_1-\phi_2\| <1/N$, then there exists a unitary $u\in D$ such 
that $\phi_1=Ad(u)\circ \phi_2$ and $\|u-I_D\| < \sqrt{2}N\|\phi_1-\phi_2\|$. 
\end{lemma}

\begin{proof}
Let $E_C^A$ be the estriction of $E_C^D$ to $A$ and $v_i=E_A^D(u_i)$. Put
\[ 
s=\sum_{i=1}^N \phi_1((\mathrm{Index}\,E_C^A)^{-1})\phi_1(v_i)\phi_2(v_i^*) .
\]
Since $\|(\mathrm{Index}\,E_C^A)^{-1}\|\le1$, 
\begin{align*}
\|s-I_D\|&=\left\| s- \sum_{i=1}^N \phi_1((\mathrm{Index}\,E_C^A)^{-1})\phi_1(v_i)\phi_1(v_i^*) \right\| \\
 &\le\sum_{i=1}^N \|\phi_1((\mathrm{Index}\, E_C^A)^{-1})\phi_1(v_i)\| \| \phi_2(v_i^*)-\phi_1(v_i^*)\| \\
 &\le N\|\phi_1-\phi_2\|<1.
\end{align*}
Therefore the unitary $u$ in the polar decomposition $s=u|s|$ lies in $D$ and satisfies $\|u-I_D\|\le \sqrt{2} N\|\phi_1-\phi_2\|$ by Lemma \ref{lem:estimates}. Furthermore, for any 
$a\in A$, 
\begin{align*}
\phi_1(a)s &=\sum_{i=1}^N \phi_1((\mathrm{Index}\,E_C^A)^{-1})\phi_1(av_i)\phi_2(v_i^*) \\
 &=\sum_{i=1}^N \sum_{j=1}^N \phi_1((\mathrm{Index}\,E_C^A)^{-1})\phi_1(v_j E_C^A(v_j^*av_i))\phi_2(v_i^*) \\
 &=\sum_{i=1}^N\sum_{j=1}^N \phi_1((\mathrm{Index}\,E_C^A)^{-1})\phi_1(v_j)E_C^A(v_j^*av_i)\phi_2(v_i^*) \\
 &=\sum_{j=1}^N\sum_{i=1}^N \phi_1((\mathrm{Index}\,E_C^A)^{-1})\phi_1(v_j)\phi_2(E_C^A(v_j^*av_i)v_i^*) \\
 &=\sum_{j=1}^N \phi_1((\mathrm{Index}\,E_C^A)^{-1})\phi_1(v_j)\phi_2(v_j^*a) =s\phi_2(a).
\end{align*}
Taking adjoints gives
\[ s^*\phi_1(a)=\phi_2(a)s^* \ , \ \ a\in A .\]
Therefore, 
\[ s^*s\phi_2(a)=s^*\phi_1(a)s=\phi_2(a)s^*s \ , \ \ a\in A .\]
Since $|s|\phi_2(a)=\phi_2(a)|s|$, $ \phi_1(a)u =u\phi_2(a)$. Therefore 
we have that  $\phi_1=Ad(u)\circ\phi_2$.
\end{proof}

The following proposition shows that sufficiently close 
intermediate subalgebras with conditional expectations are unitarily equvalent. 

\begin{prop}\label{proposition:perturbation}
Let $D$ be a $\mathrm{C}^*$-algebra and $C$ a $\mathrm{C}^*$-subalgebra of $D$ with a common unit and $E_C^D:D\to C$ a conditional expectation of finite index. Then, there exists a positive constant $\gamma$ satisfying the following: For any  $A,B\in\mathrm{IMS}(C,D,E_C^D)$, if $d(A,B)<\gamma$, then there exists a unitary $u\in C^*(A,B)$ such that $uAu^*=B$. We can choose the unitary in the relative commutant 
$C' \cap D$. 
\end{prop}

\begin{proof}
Let $N$ be  the number of  a finite basis for $E_C^D$ in $D_1$. 
Put $\gamma:= (10N)^{-4}$. By Lemma 3.3, there exists a unital $*$-homomorphism $\psi:A\to B$ such that $\psi|_C=id_C$ and 
\[ \|E_B^D|_A-\psi\|\le 8\sqrt{3}N\sqrt{d(A,B)}<8\sqrt{3}(100N)^{-1} .\]
Since 
\[ \|\psi-id_A\|\le\|\psi-E_B^D|_A\|+\|E_B^D|_A-id_A\| < 8\sqrt{3}(100N)^{-1}+2(10N)^{-4}< \frac{1}{N}  \]
by Lemma 3.2, there exists a unitary $u\in C^*(A,B)$ such that $\psi=Ad(u)$ and $\|u-1_D\|\le \sqrt{2}N(8\sqrt{3}(100N)^{-1}+2(10N)^{-4})$ by Lemma 3.4. That is $uAu^*\subset B$. 
For any $b\in  B_1$, there exists $a\in A_1$ such that $\|b-a\|\leq \gamma$ by $d(A,B)<\gamma$. Then
\begin{align*}
\|b-uau^*\| &\le\|b-a\|+\|a-uau^*\| \\
 &\le\gamma +2\|a\| \|u-1_D\| \\
 &\le (10N)^{-4} + 2\sqrt{2}N ( 8\sqrt{3}(100N)^{-1}+2(10N)^{-4} ) \\
 &\le (4\sqrt{2}+1)(10)^{-4} + 16\sqrt{6}(10)^{-2} <1 ;
\end{align*}
therefore, $d(uAu^*,B)<1$.  By Lemma \ref{lem:<1}, we have that $uAu^*=B$.
\end{proof}

The following is the main theorem of the paper. 
Thanks to Izumi's result in \cite{Izumi}, we do not need to 
assume the existence of the conditional expectations onto intermediate C$^*$-subalgebras 
apriori. We also need the notion minimality of conditional expecations 
in \cite{Ka-Wa}. Let $D$ be a simple C$^*$-algebra and $C$ a simple C$^*$-subalgebra of $D$ with a common unit and $E_C^D:D\to C$ a conditional expectation of finite index. Then there there exists a unique minimal conditional 
expectation $E_0: D\to C$, that is, Index$\, E_0 \leq $ Index$\,E$ for any 
conditional expectation $E:D\to C$ of finite index.

\begin{thm} \label{thm;simple}
Let $D$ be a simple $\mathrm{C}^*$-algebra and $C$ a simple $\mathrm{C}^*$-subalgebra of $D$ with a common unit and $E_C^D:D\to C$ a conditional expectation of finite index. 
Then, there exists a positive constant $\gamma$ satisfying the following: 
For any simple intermediate $\mathrm{C}^*$-subalgebras $A$ and $B$ for $C \subset D$,  
if  $d(A,B)<\gamma$, 
then there exists a unitary $u\in C^*(A,B)$ such that $uAu^*=B$. 
We can choose the unitary in the relative commutant 
$C' \cap D$. 

\end{thm}
\begin{proof} We may assume that $E_C^D:D\to C$ is 
a minimal conditional expectation of finite index by replacing the original 
conditional expectation by the minimal one, if necesary. Let $A$ be 
a simple interemediate C$^*$-subalgebra such that $C \subset A \subset D$. 
Since $E_C^D$ satisfies Pimsner-Popa inequality, so does the restriction 
$E_C^A : A \rightarrow C$.  By  Corollary 3.4 in Izumi \cite{Izumi}, 
$E_C^A$ also has a finite basis. 
Let $\{u_1,\cdots,u_N\}$ be a basis for $E_C^A$. 
Moreover there exists a 
conditional expectation $E_A^D: D \rightarrow A$ of finite 
index by Proposition 6.1 in Izumi \cite{Izumi}. 
In fact, we can define $E_A^D$ by 
\[ E_A^D(x)=\left( \mathrm{Index}\, E_C^A \right)^{-1} \sum_{i,j=1}^N u_iE_C^D(u_i^*xu_j)u_j^* \ , \ x\in D. \]
Replace 
$E_C^A$ and $E_A^D$ by minimal conditional expectations. 
Then the conposition $E_C^A \circ E_A^D$ is also a 
minimal conditional expectation by Theorem 3 in \cite{Ka-Wa}. 
Since the minimal conditional expectation is unique, 
$E_C^A \circ E_A^D = E_C^D$. Hence  we may assume that 
the conditional expectation $E_A^D$ satisfies the compatibility condition. 
Therefore Proposition \ref{proposition:perturbation} can be applied. 
\end{proof}

\begin{rem}\upshape
Even if  we do not assume that an intermediate C$^*$-subalgebra $A$ is not simple,  we can prove a similar fact. But the constant $\gamma$ above depends on the choice of $A$:

Let $D$ be a simple C$^*$-algebra and $C$ a simple C$^*$-subalgebra of $D$ with a common unit and $E_C^D:D\to C$ a conditional expectation of finite index.  
For any  intermediate C$^*$-subalgebras $A$,  
there exists a positive constant $\gamma$ satisfying the following: 
For any otheir intermediate C$^*$-subalgebras $B$ for  $C \subset D$,  
if  $d(A,B)<\gamma$, 
then there exists a unitary $u\in C^*(A,B)$ such that $uAu^*=B$. 
We can choose the unitary in the relative commutant 
$C' \cap D$. 
In fact, all we need is that the existance of the conditinal expectation $E^D_A$ and a finite basis for $E^A_C$.  
\end{rem}

In the case of subfactor theory of  Jones \cite{Jones}, 
the lattices of intermediate subfactors with  their finiteness 
were studied  in   Popa \cite{Po}, 
Watatani \cite{Watatani2}, Teruya-Watatani \cite{TW}, 
Longo \cite{Longo}, Khoshkam-Mashhood \cite{Kh-Ma}, Grossman-Jones \cite{G-J}, 
Grossman-Izumi \cite{G-I} and  Xu \cite{Xu} for example. In these study, 
the $\| \cdot \|_2$-perturbation technique of 
von Neumann algebras developed by 
Christensen \cite{Chris2} are essentially used or motivated.

Since we can choose the implementing unitary $u$ in the relative 
commutant, we immediately get the following finiteness of the 
intermediate subalgebras for both simple C$^*$-algebras and 
factors. A bound of the number is  obtained as in Longo \cite{Longo}. 

\begin{cor}
Let $D$ be a simple $\mathrm{C}^*$-algebra and $C$ a simple $\mathrm{C}^*$-subalgebra of $D$ with a common unit and $E_C^D:D\to C$ a conditional expectation of finite index. 
If the relative commutant $C' \cap D$ is trivial, then the number of 
intermediate $\mathrm{C}^*$-subalgebras is finite. 
\end{cor}
\begin{proof} 
Since $E_C^D:D\to C$ is a conditional expectation of finite index , 
there exists a dual conditional expectation of finite index 
$E_D:C^*\langle D,e_C\rangle \to D$. 
Thus, $E_D\circ E_C^D:C^*\langle D,e_C\rangle \to C$ is a conditional expectation of finite index. 
Since the commutant $C' \cap C$  is finite-dimensional, 
the relative commutant 
$C' \cap C^*\langle D,e_C\rangle$ is also finite-dimensional by Proposition 2.7.3 
 in \cite{Watatani1}. 
Therefore the set 
\[
\mathcal P := \{p \in C^{\prime} \cap C^*\langle D,e_C\rangle\ |\  p \text{ \ is a projection } \}
\]  
is a compact Hausdorff space with respect to the operator norm topology. 
Let $N$ be  the number of  a finite basis for $E_C^D$ in $D_1$. 
Let $\varepsilon = {(2(10N)^4\|\mathrm{Index}\ E_C^D\|})^{-1} $.
By the compactness, there exists a finite open 
covering by $\varepsilon$-open balls. For any intermediate C$^*$-subalgebras 
$A$ and $B$, their Jones projection $e_A$ and $e_B$ are in 
$C' \cap C^*\langle D,e_C\rangle$ and satisfies 
\[ 
d(A,B)\le \|\mathrm{Index}\ E_C^D\| \|e_A-e_B\|
\] 
by Lemma \ref{lem:Jones projection metric}. If two Jones projection $e_A$ and $e_B$ are in one of these 
$\varepsilon$-open balls, then $d(A,B) < (10N)^{-4}$. 
By Theorem \ref{thm;simple}, 
there exists a unitary $u$ in $C^{\prime} \cap D$ such that 
$B = uAu^*$. Since $C' \cap D = {\mathbb C}I$, $u$ is a scalar. Therefore 
$B = A$. This shows that each $\varepsilon$-open ball of the cover contains 
at most one Jones projection for some intermediate C$^*$-subalgebra. 
This completes the proof. 
\end{proof}

Since any subfactor of a type II$_1$ factor has a 
conditional expectation, we can also apply the same  method in this case.

\begin{cor}
Let $M$ be a type $\mathrm{II}_1$ factor and $N$ a subfactor of finite index. 
If the relative commutant $N'\cap M$ is trivial, then 
the set of intermediate subfactors is a finite set. 
\end{cor}

\

\footnotesize
(Shoji Ino) {\sc Department of Mathematical Sciences, Kyushu University, Motooka, Fukuoka, 819-0395, Japan.}

{\it E-mail address:} s-ino@math.kyushu-u.ac.jp 

\

(Yasuo Watatani) {\sc Department of Mathematical Sciences, Kyushu University, Motooka, Fukuoka, 819-0395, Japan.}

{\it E-mail address:} watatani@math.kyushu-u.ac.jp 
\end{document}